\documentclass[12pt,reqno]{amsart}
\usepackage{amsmath, amsfonts, amssymb, amsthm, mathtools, fullpage, 
hyperref, xcolor, enumerate}
\usepackage{verbatim}
\usepackage[capitalise]{cleveref}
\usepackage[msc-links, alphabetic]{amsrefs}
\DeclareMathOperator{\supp}{supp}



\newtheorem{thm}{Theorem}[section]

\newtheorem{lemma}[thm]{Lemma}

\newtheorem{question}[thm]{Question}
\theoremstyle{definition}
\newtheorem{remark}[thm]{Remark}

\crefname{thm}{Theorem}{Theorems}
\crefname{lemma}{Lemma}{Lemmas}

\title{Cheeger constants and bipartiteness ratios of Cayley graphons}
\author{Mugdha Mahesh Pokharanakar}
\address{Department of Mathematics, Indian Institute of Science Education and Research Bhopal, 
Bhopal Bypass Road, Bhauri, Bhopal 462066, Madhya Pradesh, India}
\curraddr{}
\email{mugdha22@iiserb.ac.in}
\thanks{}



\begin{document}

\begin{abstract}
	We show, under a certain assumption, that 
	the Cheeger constant of a connected Cayley graphon is bounded from above 
	in terms of its bipartiteness ratio, and also in terms of 
	the gap between the top of the spectrum of its Laplacian and $2$. 
	Using this result, we prove that 
	if a connected Cayley graphon $C(G,\gamma)$ is bipartite, 
	then there exists an index two subgroup $H$ of $G$ 
	that avoids the support of $\gamma$. 
\end{abstract}

\maketitle

\section{Introduction} \label{sec:intro}

The discrete Cheeger--Buser inequality, 
established by Dodziuk \cite{Dodziuk84}, Tanner \cite{Tanner-vertexCheegerEasy}, 
Alon--Milman \cite{Alon-Milman85} and Alon \cite{Alon-vertexCheeger86},
states that if the Cheeger constant of a graph is away from $0$,
then the second smallest eigenvalue of its Laplacian is away from $0$, 
and vice versa.
The dual Cheeger--Buser inequality,
established by Trevisan \cite{Trevisan-MaxCut12} and Bauer--Jost \cite{Bauer-Jost13}, 
states that if the bipartiteness ratio of a graph is away from $0$, 
then the largest eigenvalue of its Laplacian is away from $2$, and vice versa.

Breuillard, Green, Guralnick and Tao \cite{BGGT-ExpansionCayley} proved that
the gap between the largest eigenvalue of the Laplacian of 
a connected non-bipartite Cayley graph and $2$ is uniformly bounded 
from below in terms of its Cheeger constant and the degree.
A quantitative version of this result was obtained by Biswas \cite{Biswas19}. 
Moorman, Ralli and Tetali \cite{Moorman-Ralli-Tetali-Bipartiteness} 
used a different approach to improve the bounds involved.
They first proved that the Cheeger constant of a connected non-bipartite Cayley graph
can be bounded from above in terms of its bipartiteness ratio and the degree.
Then they used the dual Cheeger--Buser inequality to improve the bounds.
For a more recent upper bound on the Cheeger constant of 
a connected non-bipartite Cayley graph in terms of 
its bipartiteness ratio and the degree, see \cite{Saha-DualCheeger}.
Li and Liu \cite{Li-Liu-extremal-eigenval25} have obtained an   improved upper bound 
on the Cheeger constant of a connected non-bipartite Cayley graph 
in terms of its degree and the gap between the largest eigenvalue 
of its Laplacian and $2$ using spectral methods.

Borgs, Chayes, Lov\'asz, S\'os, Szegedy and Vesztergombi 
developed the theory of graph limits through a series of articles including
\cites{Lovasz-Szegedy-limitsofDenseGraphs06,Lovasz-etal-Counthomo06, 
Lovasz-etal-cgtseqDenseGraphs08,Lovasz-etal-cgtseqDenseGraphs12}.
They called the ``limit objects'' as graphons.
We refer to the book \cite{Lovasz-LargeNet12} by Lov\'asz for more details.
A \emph{graphon} is a quadruple $(\Omega, \mathcal{B}, \pi, W)$,
where $(\Omega, \mathcal{B}, \pi)$ is a probability space, and 
$W \colon \Omega \times \Omega \to [0,1]$ is a symmetric function
measurable with respect to the completion of the product probability space 
$(\Omega \times \Omega, \mathcal{B} \times \mathcal{B}, \pi \times \pi)$
(see \cite[p.~1599]{Borgs-Chayes-Lovasz-uniqueness-graph-limits}, for instance).
When the underlying probability space is clear from the context, 
we refer to $W$ as a graphon.

The notion of Cayley graphons was introduced by Lov\'asz and Szegedy 
\cite{Lovasz-Szegedy-CayleyGraphons15}.
Let $G$ denote a second countable compact Hausdorff topological group,
$\mathcal{B}$ denote the Borel $\sigma$-algebra on $G$, 
and $\pi$ denote the Haar measure on $(G, \mathcal{B})$.
Let $\gamma \colon G \to [0,1]$ denote a measurable function 
such that for all $x \in G$, we have $\gamma(x) = \gamma(x^{-1})$.
Define a function $W \colon G \times G \to [0,1]$ by $W(x,y) = \gamma(xy^{-1})$
for all $x,y \in G$.
The graphon $(G, \mathcal{B}, \pi, W)$ is called a \emph{Cayley graphon}
\cite[Definition 30]{Lovasz-Szegedy-CayleyGraphons15}, 
and it is denoted by $C(G,\gamma)$.

Khetan and Mj \cite{Abhishek-Mahan24} established 
the Cheeger--Buser inequality for graphons,
and the author \cite{Pokharanakar-DualCheeger26} established 
the dual Cheeger--Buser inequality for graphons. 
In this article, analogous to the work of Moorman, Ralli and Tetali 
\cite{Moorman-Ralli-Tetali-Bipartiteness},
we study a relation between the Cheeger constants 
and the bipartiteness ratios of Cayley graphons.
Then using the dual Cheeger--Buser inequality for graphons,	
we also obtain a relation between the Cheeger constant of a Cayley graphon and 
the gap between the top of the spectrum of its Laplacian and $2$.
We refer to \cref{sec:prelims} for notations, definitions and preliminaries.
We prove the following with the convention that 
the infimum of the empty set is infinity.

\begin{thm} \label{thm:mainthm}
	Let $G$ and $\gamma$ be as above,
	and $W$ denote the Cayley graphon $C(G,\gamma)$.
	Let $h_W$ and $\beta_W$ denote its Cheeger constant 
	and the bipartiteness ratio respectively,
	and $\lambda_W^{\max}$ denote the top of the spectrum of its Laplacian.
	Suppose that $W$ is connected, and that
	\[
		\mathfrak{d}_W \coloneq \inf\{\beta_W(H,H^c) \mid H 
		\text{ is an index two measurable subgroup of } G\} - \beta_W 
	\] 
	is positive. 
	Choose a positive real number $\varepsilon \leq \frac{1}{50}$ such that 
	\[
	\left( 3 \varepsilon + \frac{\varepsilon}{4} \right)
	\left( 2 \varepsilon + 4 \sqrt{\frac{\varepsilon(5 - \varepsilon)(1 + \varepsilon)}
	{2(1 - \varepsilon)}} \right) + 3 \varepsilon \leq \mathfrak{d}_W.
	\]
	Then the inequality $h_W \leq \frac{4}{\varepsilon} \beta_W$ holds. 
	In particular, if the group $G$ does not admit an index two measurable subgroup,
	then the inequality $h_W \leq 200 \beta_W$ holds, and hence, we have
	$2 - \lambda_W^{\max} \geq \frac{h_W^2}{8 \cdot 10^4}$.
\end{thm}

In \cref{thm:unbounded-ratios}, we show that, unlike graphs, 
the ratios of the Cheeger constants and the bipartiteness ratios of 
connected non-bipartite Cayley graphons cannot be bounded from above 
solely in terms of their degrees, 
and that the hypothesis in \cref{thm:mainthm} that $\mathfrak{d}_W$ is positive, 
cannot be removed.
It would be interesting to investigate 
if \cref{thm:mainthm} can be improved in the following sense.

\begin{question}
	For a connected Cayley graphon $W$, if $\mathfrak{d}_W$ is positive,
	then does the inequality $h_W \leq C \beta_W$ hold 
	for some positive real number $C$ possibly depending on the degree function of $W$?
\end{question}

Using \cref{thm:mainthm}, we show that 
if a connected Cayley graphon $C(G,\gamma)$ is bipartite, 
then there exists an index two subgroup $H$ of $G$ that avoids the support of $\gamma$.

\begin{thm} \label{thm:connected&bipartite-index2-avoiding-gamma}
	Let $G$ and $\gamma$ be as above, 
	and $W$ denote the Cayley graphon $C(G,\gamma)$.
	If $W$ is connected and bipartite, 
	then there exists an index two measurable subgroup $H$ of $G$
	such that $\pi(\supp(\gamma) \cap H) = 0$,
	where $\supp(\gamma) = \{g \in G \mid \gamma(g) \neq 0\}$.	
\end{thm}

The proofs of \cref{thm:mainthm,thm:connected&bipartite-index2-avoiding-gamma}
are given in \cref{sec:proofs}.
To prove \cref{thm:mainthm}, we consider any disjoint measurable subsets $L, R$ of $G$
which ``approximate'' the bipartiteness ratio $\beta_W$ via $\beta_W(L,R)$, 
and try to find an upper bound on 
the ratio of the Cheeger constant $h_W$ and $\beta_W(L,R)$.
We get a desired upper bound in all the cases, 
except one where a certain dichotomy holds 
for the measures of the intersections of $L$ with its translates.
In this case, we use \cref{lem:index2subgp}, 
which is proved using a technique of Fre\u{\i}man \cite{Freiman-subgroup-technique-73},
and obtain an index two measurable subgroup of $G$ with certain properties, 
which contradict the assumptions in \cref{thm:mainthm}, 
showing that this case does not arise.
This proof is motivated by the arguments in \cite{Moorman-Ralli-Tetali-Bipartiteness}.
For proving \cref{thm:connected&bipartite-index2-avoiding-gamma},
using \cref{thm:mainthm}, we get a sequence $(H_n)$ of index two measurable subgroups 
of $G$ such that the sequence $(\beta_W(H_n,H_n^c))$ converges to zero. 
If the sequence $(H_n)$ has infinitely many distinct terms, then 
the sequence of complements of the symmetric differences of distinct subgroups in $(H_n)$
leads us to conclude that the Cheeger constant of $W$ is zero, 
which contradicts the connectedness of $W$, and the result follows.

\section{Preliminaries} \label{sec:prelims}

Let $(\Omega,\mathcal{B},\pi, W)$ be a graphon. 
We will write $\mathrm{d}x$ to denote $\mathrm{d}\pi(x)$.
The \emph{degree function} $d_W \colon \Omega \to [0,1]$ of $W$ is defined as
\[
    d_W(x) = \int_\Omega W(x,y)\, \mathrm{d}y
\] 
for all $x \in \Omega$. 
If the degree function of $W$ is equal to a constant $d \in [0,1]$ a.e.,
the graphon $W$ is said to be \emph{$d$-regular}.
Define $\nu \colon \mathcal{B} \to [0,1]$ by 
\[
    \nu(A) = \int_A d_W(x)\, \mathrm{d}x
\]
for any $A \in \mathcal{B}$. 
Note that $\nu$ is a measure on $(\Omega,\mathcal{B})$.
Consider the corresponding real Hilbert space $L^2(\Omega,\nu)$ with the inner product 
\[
	\langle f,g \rangle_v = \int_{\Omega} \int_{\Omega} f(x)g(x)W(x,y)\, \mathrm{d}x\, \mathrm{d}y
\]
for all $f,g \in L^2(\Omega,\nu)$.

A graphon $W$ is said to be \emph{connected} if 
for any measurable subset $A$ of $\Omega$ with $0 < \pi(A) < 1$, 
the integral $\int_{A \times A^c} W(x,y)\, \mathrm{d}x\, \mathrm{d}y$ is positive.
If $W$ is connected, then $d_W$ is positive $\pi$-a.e.
Henceforth, assume that $W$ is a connected graphon.
Consider the bounded linear operator $T_W \colon L^2(\Omega,\nu) \to L^2(\Omega,\nu)$, defined by
\[
	(T_W f)(x) = \int_{\Omega} W(x,y)f(y)\, \mathrm{d}y
\] 
for all $f \in L^2(\Omega,\nu)$.
The \emph{Laplacian} $\Delta_W$ of the graphon $W$ is given by
\[
	(\Delta_W f)(x) = f(x) - \frac{1}{d_W(x)} (T_W f)(x)
\] 
for all $f \in L^2(\Omega,\nu)$. 
The \emph{top of the spectrum} $\lambda_W^{\max}$ of $\Delta_W$ is defined
(for instance, see \cite{Pokharanakar-DualCheeger26}) as
\[
	\lambda_W^{\max} 
	= \sup_{f \in L^2(\Omega,\nu) \setminus \{0\}} 
	\frac{\langle \Delta_W f,f\rangle_v}{\langle f,f\rangle_v}.
\] 
For any subsets $A,B \in \mathcal{B}$, write 
\[
	\eta(A \times B) = \int_{A} \int_{B} W(x,y)\, \mathrm{d}x\, \mathrm{d}y.
\]
The \emph{Cheeger constant} of a connected graphon $W$, denoted by $h_W$,
is defined (see \cite{Abhishek-Mahan24}, for instance) by
\[
	h_W = \inf_{\substack{A \in \mathcal{B} \\ 0 < \pi(A) < 1}}
	\frac{\eta(A \times A^c)}{\min\{\nu(A), \nu(A^c)\}}.
\]
Observe that                                            
\[
	h_W = \inf_{\substack{A \in \mathcal{B} \\ 0 < \nu(A) \leq \frac{\nu(\Omega)}{2}}}
	\frac{\eta(A \times A^c)}{\nu(A)}.
\] 
A graphon $W$ is said to be \emph{bipartite} if there exist 
disjoint measurable subsets $L, R$ of $\Omega$ such that $L \cup R = \Omega$, 
and that $W$ is $\pi \times \pi$-almost everywhere zero 
on $L \times L$ and $R \times R$.
Analogous to the bipartiteness ratio of a graph, 
defined by Trevisan \cite{Trevisan-MaxCut12}, 
the \emph{bipartiteness ratio} $\beta_W$ of $W$ is defined as
\[
	\beta_W = \inf_{\substack{L,R \in \mathcal{B} \\ 
	\pi(L \cup R) > 0 \\ L \cap R = \emptyset}}
	\beta_W(L,R),
\]	
where
\[
	\beta_W(L,R) = \frac{2 \eta(L \times L) + 2 \eta(R \times R) 
	+ \eta((L \cup R) \times (L \cup R)^c)}
	{2 \nu(L \cup R)}
\]
for any disjoint measurable subsets $L,R$ of $\Omega$ such that $\pi(L \cup R) > 0$.
The dual Cheeger--Buser inequality for graphons, 
obtained by arguing similarly as in the proof of 
\cite[Theorem 1.1]{Pokharanakar-DualCheeger26}, states that 
\begin{equation} \label{ineq:dualCheeger}
	\frac{\beta_W^2}{2} \leq 2 - \lambda_W^{\max} \leq 2 \beta_W.
\end{equation}

\section{Proofs of \cref{thm:mainthm,thm:connected&bipartite-index2-avoiding-gamma}} 
\label{sec:proofs}

Henceforth, $G$ denotes a second countable compact Hausdorff topological group,
and $\mathcal{B}$ denotes the Borel $\sigma$-algebra on $G$. 
Let $\pi$ denote the Haar measure on $(G, \mathcal{B})$, and
$\gamma \colon G \to [0,1]$ denote a measurable function 
such that for all $x \in G$, we have $\gamma(x) = \gamma(x^{-1})$.
It follows from the translation invariance of $\pi$ that
the Cayley graphon $C(G,\gamma)$ is $d$-regular, 
where $d = \int_{G} \gamma(y)\, \mathrm{d}y$.
The proofs in this section are motivated by the arguments in 
\cite{Moorman-Ralli-Tetali-Bipartiteness} and \cite{Saha-DualCheeger}.

If $L$ is a measurable subset of $G$ with $\pi(L)$ ``around'' $1/2$,
and if the intersection of $L$ with any of its translates is 
either ``very small'' or ``very large'', 
then the following lemma states that there exists a subgroup $H$ of $G$  
such that $L$ is ``almost like'' $H$ or $H^c$.
It is proved using arguments similar to that in \cite{Freiman-subgroup-technique-73}.

\begin{lemma} \label{lem:index2subgp}
	Let $\varepsilon, \delta$ be any positive real numbers such that 
	$\delta < \frac{1 - 3 \varepsilon}{4}$.
	Let $L$ be a measurable subset of $G$ such that 
	$\frac{1 - \varepsilon}{2} \leq \pi(L) < \frac{1 + \varepsilon}{2}$.
	Suppose that we have
	\[
		\pi(L \cap Lg) \notin (\delta \pi(L), (1 - \delta) \pi(L))
	\]
	for every $g \in G$.
	Then the group $G$ admits an index two measurable subgroup $H$ such that
	\[
		\pi(L \cap H) < \frac{1}{2} \sqrt{(\varepsilon + \delta) \pi(L)},
		\quad \text{or} \quad
		\pi(L \cap H^c) < \frac{1}{2} \sqrt{(\varepsilon + \delta) \pi(L)}.	
	\]
\end{lemma}

\begin{proof}
	Let $H \coloneq \{g \in G \mid \pi(L \cap Lg) \geq (1 - \delta) \pi(L)\}$.
	Consequently, we have $H^c = \{g \in G \mid \pi(L \cap Lg) \leq \delta \pi(L)\}$.
	Note that the identity element of the group $G$ belongs to $H$.
	Let $h_1, h_2$ be any elements of $H$.
	Using translation invariance of $\pi$, and the facts 
	$L \cap Lh_1h_2 = L \setminus (L \setminus Lh_1h_2)$
	and $L \setminus Lh_1h_2 \subseteq (L \setminus Lh_2) \cup (Lh_2 \setminus Lh_1h_2)$,
	we obtain
	\begin{align*}
		\pi(L \cap Lh_1h_2) 
		& =
		\pi(L) - \pi(L \setminus Lh_1h_2)
		\\
		& \geq
		\pi(L) - \pi(L \setminus Lh_2) - \pi(Lh_2 \setminus Lh_1h_2)
		\\
		& =
		\pi(L) - \pi(L \setminus Lh_2) - \pi(L \setminus Lh_1)
		\\
		& =
		\pi(L) - (\pi(L) - \pi(L \cap Lh_2)) - (\pi(L) - \pi(L \cap Lh_1))
		\\
		& =
		\pi(L \cap Lh_1) + \pi(L \cap Lh_2) - \pi(L)
		\\
		& \geq 
		(1 - 2 \delta) \pi(L).
	\end{align*}
	As $\delta < \frac{1 - 3 \varepsilon}{4} < \frac{1}{3}$, 
	we conclude that $\pi(L \cap Lh_1h_2) > \delta \pi(L)$, 
	and hence, the set $H$ contains $h_1h_2$.
	Also, for any $h \in H$, note that 
	$\pi(L \cap Lh^{-1}) = \pi(Lh \cap L) \geq (1 - \delta) \pi(L)$
	holds, and thus, $h^{-1}$ lies in $H$.
	This proves that $H$ is a subgroup of $G$.
	
	Further, as the function from $G \times G$ to $\mathbb{R}$, defined by 
	\[
		(g,x) \mapsto 1_{L \cap Lg}(x) = 1_L(x) 1_L(xg^{-1}),
	\] 
	is measurable, the function 
	\[
		g \mapsto \pi(L \cap Lg) = \int_G 1_{L \cap Lg}(x)\, \mathrm{d}x
	\]	 
	from $G$ to $\mathbb{R}$ is also measurable, using the Fubini--Tonelli theorem.
	Therefore, the subgroup $H$ of $G$ is measurable.
	
	Now we will show that $H$ has index two in $G$. 
	It is enough to prove that $\frac{1}{3} < \pi(H) < 1$, 
	as all cosets of $H$ in $G$ are of the same $\pi$-measure.
	Observe that
	\begin{align*}
		\int_{G} \pi(L \cap Lg)\, \mathrm{d}g
		& =
		\int_{G} \left( \int_{G} 1_L(x) 1_L(xg^{-1})\, \mathrm{d}x \right) \mathrm{d}g
		\\
		& =
		\int_{G} \left( \int_{G} 1_L(xg^{-1})\, \mathrm{d}g \right) 1_L(x)\, \mathrm{d}x
		\\
		& =
		\int_{G} \left( \int_{G} 1_L(g)\, \mathrm{d}g \right) 1_L(x)\, \mathrm{d}x
		\\
		& =
		\pi(L)^2,
	\end{align*}
	and that
	\[
	\int_{G} \pi(L \cap Lg)\, \mathrm{d}g 
	= \int_{H} \pi(L \cap Lg)\, \mathrm{d}g + \int_{H^c} \pi(L \cap Lg)\, \mathrm{d}g
	\leq \pi(L) \pi(H) + \delta \pi(L) \pi(H^c).
	\] 
	It follows that $\pi(L) \leq (1 - \delta) \pi(H) + \delta$, which implies
	\[
	\pi(H) \geq \frac{\pi(L) - \delta}{1 - \delta} 
	\geq \frac{\frac{1 - \varepsilon}{2} - \delta}{1 - \delta} > \frac{1}{3}, 
	\] 
	where the last inequality follows from the hypothesis that 
	$\delta < \frac{1 - 3 \varepsilon}{4}$.
	If $\pi(H) = 1$, then we have
	\[
		\pi(L)^2 = \int_{G} \pi(L \cap Lg)\, \mathrm{d}g 
		= \int_{H} \pi(L \cap Lg)\, \mathrm{d}g
		\geq (1 - \delta) \pi(L).
        \]
	This implies that
	\[
		\pi(L) \geq 1 - \delta > 1 - \frac{1 - 3 \varepsilon}{4}
		> \frac{1 + \varepsilon}{2},
	\]
	which is a contradiction. 
	So, we conclude that $\frac{1}{3} < \pi(H) < 1$, and consequently, 
	the index of $H$ in $G$ is two.

	Using the facts that $HH = H^c H^c = H$ and $HH^c = H^c H = H^c$, 
	and the translation invariance of $\pi$, we obtain
	\begin{align*}
		\int_{H} \pi(L \cap Lg)\, \mathrm{d}g
		& =
		\int_{H} \left( \int_{L} 1_L(xg^{-1})\, \mathrm{d}x \right) \mathrm{d}g
		\\
		& =
		\int_{L} \left( \int_{H} 1_L(xg^{-1})\, \mathrm{d}g \right) \mathrm{d}x
		\\
		& =
		\int_{L \cap H} 
		\left( \int_{H} 1_{L \cap H}(xg^{-1})\, \mathrm{d}g \right) \mathrm{d}x
		+ \int_{L \cap H} 
		\left( \int_{H} 1_{L \cap H^c}(xg^{-1})\, \mathrm{d}g \right) \mathrm{d}x
		\\
		& \hspace{0.46cm} +
		\int_{L \cap H^c} 
		\left( \int_{H} 1_{L \cap H}(xg^{-1})\, \mathrm{d}g \right) \mathrm{d}x
		+ \int_{L \cap H^c} 
		\left( \int_{H} 1_{L \cap H^c}(xg^{-1})\, \mathrm{d}g \right) \mathrm{d}x
		\\
		& =
		\int_{L \cap H} 
		\left( \int_{H} 1_{L \cap H}(xg^{-1})\, \mathrm{d}g \right) \mathrm{d}x
		+ \int_{L \cap H^c} 
		\left( \int_{H} 1_{L \cap H^c}(xg^{-1})\, \mathrm{d}g \right) \mathrm{d}x
		\\
		& =
		\int_{L \cap H} 
		\left( \int_{G} 1_{L \cap H}(xg^{-1})\, \mathrm{d}g \right) \mathrm{d}x
		+ \int_{L \cap H^c} 
		\left( \int_{G} 1_{L \cap H^c}(xg^{-1})\, \mathrm{d}g \right) \mathrm{d}x
		\\
		& =
		\pi(L \cap H)^2 + \pi(L \cap H^c)^2
		\\
		& =
		\left( \pi(L \cap H) + \pi(L \cap H^c) \right)^2 
		- 2 \pi(L \cap H) \pi(L \cap H^c)
		\\
		& =
		\pi(L)^2 - 2 \pi(L \cap H) \pi(L \cap H^c)
		\\
		& <
		\frac{1 + \varepsilon}{2} \pi(L) - 2 \pi(L \cap H) \pi(L \cap H^c).
	\end{align*}
	Combining the above with the inequality
	\[
		\int_{H} \pi(L \cap Lg)\, \mathrm{d}g \geq (1 - \delta) \pi(L) \pi(H)
		= \frac{1 - \delta}{2} \pi(L),	
	\]
	we have
	\[
		\pi(L \cap H) \pi(L \cap H^c) < \frac{\varepsilon + \delta}{4} \pi(L),
	\]
	and as a consequence, the desired result follows.
\end{proof}

Next, we prove \cref{thm:mainthm} using \cref{lem:index2subgp}.

\begin{proof}[Proof of \cref{thm:mainthm}]
	If $\beta_W \geq \frac{\varepsilon}{4}$, 
	then we are done using the fact that $h_W \leq 1$. 
	Henceforth, assume that $\beta_W < \frac{\varepsilon}{4}$.
	Let $\epsilon$ be any positive real number with $\epsilon \leq \varepsilon$.	
	Choose disjoint measurable subsets $L, R$ of $G$ such that 
	$\pi(L \cup R) > 0$ and $\beta_W(L,R) - \beta_W < \epsilon$.
	It suffices to show that $h_W \leq \frac{4}{\varepsilon} \beta_W(L,R)$.
	
	First, suppose that $\pi(L \cup R) < 1 - \varepsilon$.
	If $\pi(L \cup R) \leq \frac{1}{2}$, then note that
	\[
		h_W \leq \frac{\eta((L \cup R) \times (L \cup R)^c)}{\nu(L \cup R)}
		\leq 2 \beta_W(L,R).
	\] 
	Otherwise, we have $\pi((L \cup R)^c) < \frac{1}{2}$, and thus, we get
	\begin{align*}
		h_W 
		& \leq
		\frac{\eta((L \cup R) \times (L \cup R)^c)}{\nu((L \cup R)^c)}
		\\
		& \leq
		\frac{2 \beta_W(L,R) \nu(L \cup R)}{\nu((L \cup R)^c)}
		\\
		& <
	    \frac{2}{\varepsilon} \beta_W(L,R).
	\end{align*}

	Now suppose that $\pi(L \cup R) \geq 1 - \varepsilon$. 
	We assume that $\pi(L) \geq \pi(R)$ (by interchanging $L$ and $R$, if required). 
	Then we have $\pi(L) \geq \frac{1 - \varepsilon}{2}$.

	Let us consider the case that $\pi(L) \geq \frac{1 + \varepsilon}{2}$.
	We get $\pi(R) \leq 1 - \pi(L) \leq \frac{1 - \varepsilon}{2}$.
	This implies that
	\begin{align*}
		\eta(L \times R^c)
		& \geq
		\eta(L \times G)  - \eta(G \times R)
		\\
		& \geq
		d (\pi(L) - \pi(R))
		\\
		& \geq
		d \varepsilon
		\\
		& \geq
		d \varepsilon h_W.
	\end{align*}
	Also, we have
	\[
		\eta(L \times R^c) 
		\leq \eta(L \times L) + \eta((L \cup R) \times (L \cup R)^c)
		\leq 2d \beta_W(L, R),
	\] 
	and hence, it follows that $h_W \leq \frac{2}{\varepsilon} \beta_W(L,R)$.

	Now it remains to consider the case that 
	$\frac{1 - \varepsilon}{2} \leq \pi(L) < \frac{1 + \varepsilon}{2}$.
	Let $g \in G$ be arbitrary.
	Following the proof of \cite[Theorem 3.1]{Saha-DualCheeger}, it follows that 
	\begin{align*}
		& \hspace{0.56cm} 
		((L \cap Lg) \cup (R \cap Rg)) \times ((L \cap Lg) \cup (R \cap Rg))^c
		\\
		& \subseteq
		(L \times R^c) \cup (R \times L^c) 
		\cup (Lg \times (Rg)^c) \cup (Rg \times (Lg)^c)
		\\
		& =
		(L \times R^c) \cup (R \times L^c) 
		\cup (Lg \times R^c g) \cup (Rg \times L^c g),
	\end{align*}
	and similarly,
	\begin{align*}
		& \hspace{0.56cm} 
		((L \cap Rg) \cup (R \cap Lg)) \times ((L \cap Rg) \cup (R \cap Lg))^c
		\\
		& \subseteq
		(L \times R^c) \cup (R \times L^c) 
		\cup (Rg \times L^c g) \cup (Lg \times R^c g).
	\end{align*}
	Then the translation invariance of $\pi$ implies
	\begin{align*}
		& \hspace{0.56cm}
		\eta(((L \cap Lg) \cup (R \cap Rg)) \times ((L \cap Lg) \cup (R \cap Rg))^c)
		\\
		& \leq
		\eta(L \times R^c) + \eta(R \times L^c) 
		+ \eta(Lg \times R^c g) + \eta(Rg \times L^c g)
		\\
		& =
		2 \eta(L \times L) + 2 \eta(L \times (L \cup R)^c)
		+ 2 \eta(R \times R) + 2 \eta(R \times (L \cup R)^c)
		\\
		& \leq
		4d \beta_W(L,R) \pi(L \cup R)
		\\
		& \leq
		4d \beta_W(L,R).	
	\end{align*}	
	Similarly, we have
	\[
		\eta(((L \cap Rg) \cup (R \cap Lg)) \times ((L \cap Rg) \cup (R \cap Lg))^c)
		\leq 4d \beta_W(L,R).
	\]
	Note that the sets 
	$(L \cap Lg) \cup (R \cap Rg)$ and $(L \cap Rg) \cup (R \cap Lg)$ are disjoint, 
	and therefore, one of them has $\pi$-measure at most $\frac{1}{2}$.
	It follows that 
	\begin{equation} \label{ineq:minpih-bdd}
		\min\{\pi((L \cap Lg) \cup (R \cap Rg)), \pi((L \cap Rg) \cup (R \cap Lg))\} h_W
		\leq 4 \beta_W(L,R).
	\end{equation}
	
	Now we bound from below the left-hand-side of the above inequality.
	Let $\delta = \frac{4 \varepsilon}{1 - \varepsilon}$. 
	Using the fact that $\varepsilon \leq \frac{1}{20}$, 
	we have $\delta < \frac{1 - 3 \varepsilon}{4}$.
	We claim that there is an element $g$ of $G$ such that 
	$\pi(L \cap Lg) \in (\delta \pi(L), (1 - \delta) \pi(L))$.
	Suppose this is not the case. 
	Then by \cref{lem:index2subgp}, 
	we get an index two measurable subgroup $H$ of $G$ such that
	\[
		\pi(L \cap H) < \frac{1}{2} \sqrt{(\varepsilon + \delta) \pi(L)},
		\quad \text{or} \quad
		\pi(L \cap H^c) < \frac{1}{2} \sqrt{(\varepsilon + \delta) \pi(L)}.	
	\] 
	If $\pi(L \cap H) < \frac{1}{2} \sqrt{(\varepsilon + \delta) \pi(L)}$,
	then denote $H$ by $X$. Otherwise, denote $H^c$ by $X$. 
	Observe that
	\[
		\pi(L \cap X) 
		< \frac{1}{2} \sqrt{\frac{(\varepsilon + \delta)(1 + \varepsilon)}{2}},
		\quad \text{and} \quad
		\pi(L \cap X^c) 
		> \pi(L) - \frac{1}{2} \sqrt{\frac{(\varepsilon + \delta)(1 + \varepsilon)}{2}},
	\]
	and that
	\[
		\pi(L^c \cap X^c) = \frac{1}{2} - \pi(L \cap X^c)
		< \frac{\varepsilon}{2} 
		+ \frac{1}{2} \sqrt{\frac{(\varepsilon + \delta)(1 + \varepsilon)}{2}}.
	\]
	Using that $XX^{-1} = X^c (X^c)^{-1} = H$ and $\pi(X) = \pi(X^c) = \frac{1}{2}$, 
	we obtain
	\[
		\eta(X \times X) = \eta(X^c \times X^c) = \frac{1}{2}\int_{H} \gamma(y)\, \mathrm{d}y, 
	\]
	and hence, $\beta_W(H,H^c) = \frac{1}{d} \int_{H} \gamma(y)\, \mathrm{d}y$.
	Also, we have
	\begin{align*}
		\eta(X \times X) 
		& =
		\eta((X \cap L) \times X) + \eta((X \cap L^c) \times (X \cap L)) 
		+ \eta((X \cap L^c) \times (X \cap L^c))
		\\
		& \leq
		2 \eta((X \cap L) \times X) + \eta((X \cap L^c) \times (X \cap L^c)) 
		\\
		& \leq 
		2 \int_{G} 
		\left( \int_{G} \gamma(yx^{-1}) 1_{X}(y)\, \mathrm{d}y \right) 
		1_{X \cap L}(x)\, \mathrm{d}x + \eta(L^c \times L^c)
		\\
		& =
		2 \int_{G} 
		\left( \int_{G} \gamma(y) 1_{X}(yx)\, \mathrm{d}y \right) 
		1_{X \cap L}(x)\, \mathrm{d}x + \eta(L^c \times L^c)
        \\
		& =
		2 \pi(X \cap L) \int_{H} \gamma(y)\, \mathrm{d}y + \eta(L^c \times L^c),
		\tag{using $HX = X$ and $H^c X = X^c$}
	\end{align*}
	and similarly,
	\[
		\eta(X^c \times X^c) 
		\leq 2 \pi(X^c \cap L^c) \int_{H} \gamma(y)\, \mathrm{d}y + \eta(L \times L). 
	\]
	It follows that
	\begin{align*}
		d \beta_W(H,H^c)
		& =
		\eta(X \times X) + \eta(X^c \times X^c) 
		\\
		& \leq
		2 \left( \pi(X \cap L) + \pi(X^c \cap L^c) \right) 
		\int_{H} \gamma(y)\, \mathrm{d}y
		+ \eta(L \times L) + \eta(L^c \times L^c)
		\\
		& \leq
		\left( \varepsilon + 2 \sqrt{\frac{(\varepsilon + \delta)(1 + \varepsilon)}{2}} \right)
		d \beta_W(H,H^c)
		\\
		& \hspace{0.46cm} + 
		\eta(L \times L) + \eta(R \times R) + \eta((L \cup R)^c \times R) 
		+ \eta(L^c \times (L \cup R)^c)
		\\
		& \leq
		\left( \varepsilon + 2 \sqrt{\frac{(\varepsilon + \delta)(1 + \varepsilon)}{2}} \right)
		d \beta_W(H,H^c)
		\\
		& \hspace{0.46cm} +	
		\eta(L \times L) + \eta(R \times R) + 2d \pi((L \cup R)^c) 
		\\
		& \leq 
		\left( \varepsilon + 2 \sqrt{\frac{(\varepsilon + \delta)(1 + \varepsilon)}{2}} \right)
		d \beta_W(H,H^c) + d \beta_W(L,R) + 2d \varepsilon
		\\
		& <
		\left( \varepsilon + 2 \sqrt{\frac{(\varepsilon + \delta)(1 + \varepsilon)}{2}} \right)
		d \beta_W(H,H^c) + 2d \varepsilon + d (\beta_W + \epsilon), 
	\end{align*}	
	and using the inequality $\epsilon \leq \varepsilon$, we get
	\[
		\left( 
		1 - \varepsilon - 2 \sqrt{\frac{(\varepsilon + \delta)(1 + \varepsilon)}{2}} 
		\right) 
		\beta_W(H,H^c) 
		< \beta_W + 3 \varepsilon,  
	\]	
	that is,	
	\[
		\left( 1 - \varepsilon 
		- 2 \sqrt{\frac{\varepsilon(5 - \varepsilon)(1 + \varepsilon)}{2(1 - \varepsilon)}} 
		\right) \beta_W(H,H^c) 
		< \beta_W + 3 \varepsilon.
	\]	
	This implies 
	\begin{align*}
		\beta_W(H,H^c) 
		& < 
		(\beta_W + 3 \varepsilon) \left( 1 - \varepsilon 
		- 2 \sqrt{\frac{\varepsilon(5 - \varepsilon)(1 + \varepsilon)}{2(1 - \varepsilon)}} \right)^{-1}
		\\
		& \leq
		(\beta_W + 3 \varepsilon) \left( 1 + 2 \left( \varepsilon 
		+ 2 \sqrt{\frac{\varepsilon(5 - \varepsilon)(1 + \varepsilon)}{2(1 - \varepsilon)}} \right) \right)
		\tag{using $\varepsilon \leq \frac{1}{50}$, and 
		$\frac{1}{1-x} \leq 1 + 2x$ for $0 \leq x \leq \frac{1}{2}$}
		\\
		& <
		\beta_W + \left( 3 \varepsilon + \frac{\varepsilon}{4} \right)
		\left( 2 \varepsilon + 4 \sqrt{\frac{\varepsilon(5 - \varepsilon)(1 + \varepsilon)}
		{2(1 - \varepsilon)}} \right) + 3 \varepsilon 
		\tag{using $\beta_W < \frac{\varepsilon}{4}$},	
	\end{align*}
	and hence, we get $\beta_W(H,H^c) - \beta_W < \mathfrak{d}_W$, which is a contradiction.

	Now choose $g \in G$ such that $\pi(L \cap Lg)$ lies in the interval
	$(\delta \pi(L), (1 - \delta) \pi(L))$.
	Then, we have
	\[
		\pi((L \cap Lg) \cup (R \cap Rg)) \geq \pi(L \cap Lg) > \delta \pi(L)
		\geq \frac{\delta (1 - \varepsilon)}{2},
	\]
	and the fact that
	$L \setminus Lg \subseteq (L \setminus (Rg)^c) \cup ((Rg)^c \setminus Lg)$ implies
	\begin{align*}
		\pi((L \cap Rg) \cup (R \cap Lg)) 
		& \geq 
		\pi(L \setminus (Rg)^c)
		\\
		& \geq
		\pi(L \setminus Lg) - \pi((Rg)^c \setminus Lg)
		\\
		& =
		\pi(L) - \pi(L \cap Lg) - \pi((L \cup R)^c)
		\\
		& \geq
		\pi(L) - (1 - \delta) \pi(L) - \pi((L \cup R)^c)
		\\
		& \geq
		\frac{\delta (1 - \varepsilon)}{2}  - \varepsilon.
	\end{align*}
	This yields
	\[
		\min\{\pi((L \cap Lg) \cup (R \cap Rg)), \pi((L \cap Rg) \cup (R \cap Lg))\}
		\geq \frac{\delta (1 - \varepsilon)}{2}  - \varepsilon = \varepsilon.
	\]
	Combining this with \eqref{ineq:minpih-bdd} gives 
	$\varepsilon h_W \leq 4 \beta_W(L,R)$.
	Considering all the above cases, we conclude that 
	$h_W \leq \frac{4}{\varepsilon} \beta_W(L,R)$.

	Note that if the group $G$ does not admit an index two measurable subgroup,
	then taking $\varepsilon = \frac{1}{50}$ yields 
	the inequality $h_W \leq 200 \beta_W$.
	Combining this with the dual Cheeger--Buser inequality \eqref{ineq:dualCheeger}, 
	we obtain
	\[
		2 - \lambda_W^{\max} \geq \frac{h_W^2}{2 (200)^2}.
		\qedhere
	\]	
\end{proof}

Now, we give a proof of \cref{thm:connected&bipartite-index2-avoiding-gamma}
using \cref{thm:mainthm}.

\begin{proof}[Proof of \cref{thm:connected&bipartite-index2-avoiding-gamma}]
	Let $W = C(G,\gamma)$ be a connected bipartite Cayley graphon.
	Then we have $\beta_W = 0$, and since $W$ is $d$-regular with $d > 0$, 
	using the arguments similar to those in the proof of 
	\cite[Proposition 7.7]{Abhishek-Mahan24}, it follows that $h_W > 0$.
	Using \cref{thm:mainthm}, we get a sequence $(H_n)_n$ 
	of index two measurable subgroups of $G$ such that 
	the sequence $(\beta_W(H_n,H_n^c))_n$ converges to zero.
	Suppose that the sequence $(H_n)$ has infinitely many distinct terms.
	Choose a subsequence $(H_{n_k})_k$ of $(H_n)$ such that 
	for every positive integer $k$, the subgroups $H_{n_k}$ and $H_{n_{k+1}}$ are distinct.
	By abuse of notation, we denote $H_{n_k}$ by $H_k$ for any $k$.

	Let $k$ be an arbitrary positive integer. 
	Set $\mathcal{H}_k = (H_k \cap H_{k+1}) \cup (H_k^c \cap H_{k+1}^c)$.
	Note that $\mathcal{H}_k$ is a measurable subset of $G$ with $0 < \pi(\mathcal{H}_k) < 1$.
	Indeed, using translation invariance of $\pi$, we have 
	\[
		\pi(H_k \cap H_{k+1}) = \pi(H_k \cap H_{k+1}^c) 
		= \pi(H_k^c \cap H_{k+1}) = \pi(H_k^c \cap H_{k+1}^c),
	\]	 
	and hence, $\pi(\mathcal{H}_k) = \frac{1}{2}$.
	Observe that
	\begin{align*}
		\eta(\mathcal{H}_k \times \mathcal{H}_k^c) 
		& =
		\eta((H_k \cap H_{k+1}) \times (H_k \cap H_{k+1}^c))
		+ \eta((H_k \cap H_{k+1}) \times (H_k^c \cap H_{k+1})) 
		\\
		& \hspace{0.44cm}
		+ \eta((H_k^c \cap H_{k+1}^c) \times (H_k \cap H_{k+1}^c))
		+ \eta((H_k^c \cap H_{k+1}^c) \times (H_k^c \cap H_{k+1}))
		\\
		& \leq
		\eta(H_k \times H_k) + \eta(H_{k+1} \times H_{k+1})
		+ \eta(H_{k+1}^c \times H_{k+1}^c) + \eta(H_k^c \times H_k^c)
		\\
		& =
		d(\beta_W(H_k,H_k^c) + \beta_W(H_{k+1},H_{k+1}^c)).
	\end{align*}
	It follows that the sequence $(\eta(\mathcal{H}_k \times \mathcal{H}_k^c))_k$ 
	converges to zero, and thus, $h_W = 0$, which is a contradiction.
	This proves that the sequence $(H_n)$ has only finitely many distinct terms.
	Consequently, we get an index two measurable subgroup $H$ of $G$ with 
	\[
		\beta_W(H,H^c) = \frac{1}{d} \int_{H} \gamma(y)\, \mathrm{d}y = 0,
	\]
	which implies $\pi(\supp(\gamma) \cap H) = 0$. 
\end{proof}

\section{A contrast between graphs and graphons} \label{sec:example}

Let $\Gamma$ be a connected non-bipartite Cayley graph with degree $d$. 
Let $h_{\Gamma}$ and $\beta_{\Gamma}$ denote its Cheeger constant and 
the bipartiteness ratio respectively.
Saha \cite[Theorem 1.1]{Saha-DualCheeger} proved that 
$h_{\Gamma} \leq 90d \beta_{\Gamma}$.
In this section, we prove that no such bounds exist 
for connected non-bipartite Cayley graphons.

\begin{thm} \label{thm:unbounded-ratios}
	The ratios of the Cheeger constants and the bipartiteness ratios 
	of connected non-bipartite Cayley graphons cannot be bounded from above
	by a function of their degrees.
\end{thm}

Let $G$ be a topological group as in \cref{sec:intro}.
Assume that the Haar probability measure $\pi$ on $G$ is non-atomic.
This ensures that given any measurable subset of $G$ with $\pi(A) > 0$ 
and a real number $t$ satisfying $0 \leq t \leq \pi(A)$, 
there exists a measurable subset $B$ of $A$ such that $\pi(B) = t$
(for instance, see \cite[Lemma A.1]{Janson-Graphons-CutNorm}).
Suppose that $G$ admits an index two measurable subgroup $H$.
Let $a, b$ be positive real numbers not exceeding $1$.
Define a function $\gamma \colon G \to [0,1]$ by
\[
	\gamma(x) = 
	\begin{cases}
		a & \text{if } x \in H,
		\\
		b & \text{if } x \in H^c.
	\end{cases}
\]
Note that $\gamma$ is a symmetric measurable function. 
Denote the Cayley graphon $C(G,\gamma)$ by $W_{a,b}$.
Note that $W_{a,b}$ is a $d$-regular graphon with $d = \frac{1}{2} (a + b)$.
We call $d$ the \emph{degree} of the graphon $W_{a,b}$.

\begin{lemma} \label{lem:compute-h}
	The graphon $W_{a,b}$ is connected with 
	\[
		h_{W_{a,b}} = 
		\begin{cases}
			\frac{b}{a + b} & \text{if } a \geq b,
			\\
			\frac{1}{2} & \text{if } a < b.
		\end{cases}
	\]
\end{lemma}

\begin{proof}
	Let $A$ be any measurable subset of $G$ with $0 < \pi(A) < 1$. Then, we have
	\begin{align*}
		\eta(A \times A^c) 
		& = 
		\eta((A \cap H) \times (A^c \cap H)) + \eta((A \cap H) \times (A^c \cap H^c))
		\\
		& \quad +
		\eta((A \cap H^c) \times (A^c \cap H)) + \eta((A \cap H^c) \times (A^c \cap H^c))
		\\
		& =
		a \left( \pi(A \cap H) \pi(A^c \cap H) + \pi(A \cap H^c) \pi(A^c \cap H^c) \right) 
		\\
		& \quad +
		b \left( \pi(A \cap H) \pi(A^c \cap H^c) + \pi(A \cap H^c) \pi(A^c \cap H) \right)
		\\ 
		& \geq
		b \left( \pi(A \cap H) \pi(A^c \cap H^c) + \pi(A \cap H^c) \pi(A^c \cap H) \right).	
	\end{align*}
	Since $\pi(A) > 0$, at least one of $\pi(A \cap H)$ and $\pi(A \cap H^c)$ is positive.
	Also, at least one of $\pi(A^c \cap H)$ and $\pi(A^c \cap H^c)$ is positive, as $\pi(A) < 1$.
	Similarly, as $\pi(H), \pi(H^c) > 0$, at least one of $\pi(A \cap H)$ and $\pi(A^c \cap H)$,
	and at least one of $\pi(A \cap H^c)$ and $\pi(A^c \cap H^c)$ are positive.
	This, along with the assumption that $b$ is positive, 
	implies that $\eta(A \times A^c)$ is positive.
	Thus, the graphon $W_{a,b}$ is connected. 
	
	Further, note that
	\begin{align} \label{eq:eta-A-A^c}
		& \hspace{0.55cm}
		\eta(A \times A^c) \nonumber
		\\
		& =
		a \left( \pi(A \cap H) \left( \frac{1}{2} - \pi(A \cap H) \right)  
		+ (\pi(A) - \pi(A \cap H)) \left( \frac{1}{2} - \pi(A) + \pi(A \cap H) \right) \right)
		\nonumber
		\\
		& \quad +
		b \left( \pi(A \cap H) \left( \frac{1}{2} - \pi(A) + \pi(A \cap H) \right) 
		+ (\pi(A) - \pi(A \cap H)) \left( \frac{1}{2} - \pi(A \cap H) \right) \right)
		\nonumber
		\\
		& =
		a \left( \frac{1}{2} \pi(A) - \pi(A)^2 + 2 \pi(A) \pi(A \cap H) 
		- 2 \pi(A \cap H)^2 \right)
		\nonumber
		\\
		& \quad + 
		b \left( \frac{1}{2} \pi(A) + 2 \pi(A \cap H)^2 - 2 \pi(A) \pi(A \cap H) \right)
		\nonumber
		\\
		& =
		\frac{1}{2} (a + b) \pi(A) 
		- 2(a - b) \left( \pi(A \cap H)^2 - \pi(A) \pi(A \cap H) \right) - a \pi(A)^2.
	\end{align}
	Hence, we obtain
	
	\begin{align*}
		h_{W_{a,b}} 
		& = 
		\inf_{\substack{A \in \mathcal{B} \\ 0 < \pi(A) \leq \frac{1}{2}}} 
		\frac{\eta(A \times A^c)}{d \pi(A)}
		\\
		& =
		\inf_{\substack{A \in \mathcal{B} \\ 0 < \pi(A) \leq \frac{1}{2}}} 
		\frac{\frac{1}{2} (a + b) \pi(A) 
		- 2(a - b) \left( \pi(A \cap H)^2 - \pi(A) \pi(A \cap H) \right) - a \pi(A)^2}
		{\frac{1}{2} (a + b) \pi(A)}
		\\
		& =
		\inf_{\substack{0 \leq x \leq y \\ 0 < y \leq \frac{1}{2}}}	
		\frac{\frac{1}{2} (a + b)y - 2 (a - b)(x^2 - xy) - a y^2}{\frac{1}{2} (a + b)y}
		\tag{using non-atomicity of $\pi$}
		\\
		& =
		1 - 2 \sup_{\substack{0 \leq x \leq y \\ 0 < y \leq \frac{1}{2}}}	
		\frac{2(a - b)(x^2 - xy) + a y^2}{(a + b)y}.
	\end{align*}
	Suppose that $a \geq b$.
	Then, for any real numbers $x, y$ with $0 \leq x \leq y$ and $0 < y \leq \frac{1}{2}$,
	the inequality 
	\[
		\frac{2(a - b)(x^2 - xy) + a y^2}{(a + b)y} \leq \frac{a}{a + b} y 
		\leq \frac{a}{2(a + b)}
	\]
	holds. 
	Moreover, for $x = 0$ and $y = \frac{1}{2}$, the equality holds in the above inequality.
	So, in this case, we have
	\[
		h_{W_{a,b}} = 1 - \frac{a}{a + b} = \frac{b}{a + b}.
	\]
	Now assume that $a < b$.
	For any real numbers $x, y$ with $0 \leq x \leq y$ and $0 < y \leq \frac{1}{2}$,
	we get
	\begin{align*}
		\frac{2(a - b)(x^2 - xy) + a y^2}{(a + b)y}
		& = 
		\frac{y}{a + b} \left( a + 2(a - b) \left( \frac{x^2}{y^2} - \frac{x}{y} \right) \right)
		\\
		& \leq
		\frac{y}{a + b} \left( a - \frac{a - b}{2} \right)		
		\\
		& =
		\frac{y}{2}
		\\
		& \leq
		\frac{1}{4}.
	\end{align*}
	Note that the equality holds in the above inequalities 
	for $x = \frac{1}{4}$ and $y = \frac{1}{2}$. 
	This shows that $h_{W_{a,b}} = \frac{1}{2}$ when $a < b$.
\end{proof}

\begin{lemma} \label{lem:compute-beta}
	The graphon $W_{a,b}$ is non-bipartite with
	\[
		\beta_{W_{a,b}} = 
		\begin{cases}
			\frac{1}{2} & \text{if } a \geq b,
			\\
			\frac{a}{a + b} & \text{if } a < b.
		\end{cases}
	\]
	In particular, we have $\beta_{W_{a,b}} = \beta_{W_{a,b}}(H,H^c)$ if $a < b$.
\end{lemma}

\begin{proof}
	Let $A$ be a measurable subset of $G$ with $\pi(A) > 0$.
	Then, we get
	\begin{align*}
		\eta(A \times A) 
		& =
		\eta((A \cap H) \times (A \cap H)) + \eta((A \cap H) \times (A \cap H^c))
		\\
		& \quad +
		\eta((A \cap H^c) \times (A \cap H)) + \eta((A \cap H^c) \times (A \cap H^c))
		\\
		& =
		a \left( \pi(A \cap H)^2 + \pi(A \cap H^c)^2 \right) + 2b \pi(A \cap H) \pi(A \cap H^c)
		\\
		& \geq
		a \left( \pi(A \cap H)^2 + \pi(A \cap H^c)^2 \right),	
	\end{align*}
	which is positive, since $a$ is positive and 
	at least one of $\pi(A \cap H)$ and $\pi(A \cap H^c)$ is positive.
	It follows that the graphon $W_{a,b}$ is not bipartite.
	
	We will now compute the bipartiteness ratio $\beta_{W_{a,b}}$ of $W_{a,b}$.
	Let $L, R$ be disjoint measurable subsets of $G$ with $\pi(L \cup R) > 0$.
	Combining the above with \cref{eq:eta-A-A^c} gives 
	\begin{align*}
		& \hspace{0.56cm}
		2 \eta(L \times L) + 2 \eta(R \times R) + \eta((L \cup R) \times (L \cup R)^c)
		\\
		& =
		2a \left( \pi(L \cap H)^2 + \pi(L \cap H^c)^2 
		+ \pi(R \cap H)^2 + \pi(R \cap H^c)^2 \right) 
		\\
		& \quad +
		4b \left( \pi(L \cap H) \pi(L \cap H^c) + \pi(R \cap H) \pi(R \cap H^c) \right)	
		+ \frac{1}{2} (a + b) \pi(L \cup R) 
		\\
		& \quad -
		2(a - b) \left( \pi((L \cup R) \cap H)^2 - \pi(L \cup R) \pi((L \cup R) \cap H) \right) 
		- a \pi(L \cup R)^2.	
	\end{align*}
	Also note that
	\[
		2 \eta((L \cup R) \times G) = 2d \pi(L \cup R) = (a + b) \pi(L \cup R).
	\]
	
	Consider the set 
	\[
		S \coloneq \left\{ (x,y,z,w) \in \mathbb{R}^4 \biggm| 
		\begin{array}{l}
			x \geq 0, y \geq 0, 0 < x + y \leq 1,
			\\ 
			0 \leq z \leq x, 0 \leq w \leq y, 
			x + y - \frac{1}{2} \leq z + w \leq \frac{1}{2} 
		\end{array}
		\right\},
	\]
	and for any element $(x,y,z,w)$ of $S$, let 
	\begin{align*}
		u(x,y,z,w) 
		& = 
		2a \left( z^2 + (x - z)^2 + w^2 + (y - w)^2 \right) 
		+ 4b \left( z(x - z) + w (y - w) \right)
		\\
		& \quad -
		2 (a - b) \left( (z + w)^2 - (x + y)(z + w) \right) - a (x + y)^2.
	\end{align*}
	Then, observe that
	\begin{align*}
		\beta_{W_{a,b}} 
		& = 
		\inf_{\substack{L,R \in \mathcal{B} \\ 
		\pi(L \cup R) > 0 \\ L \cap R = \emptyset}} \beta_{W_{a,b}}(L,R) 
		\\
		& = 
		\inf_{(x,y,z,w) \in S} \frac{\frac{1}{2} (a + b)(x + y) + u(x,y,z,w)}{(a + b)(x + y)}
		\\
		& =
		\frac{1}{2} + \frac{1}{a + b} \inf_{(x,y,z,w) \in S} \frac{u(x,y,z,w)}{x + y}. 
	\end{align*}
	Let $(x,y,z,w)$ be any element of $S$. We have
	\begin{align*}
		\frac{u(x,y,z,w)}{x + y} 
		& =
		\frac{a(x - y)^2 + 2(a - b) \left( (z - w)^2 - (x - y)(z - w) \right)}{x + y}
		\\
		& =
		\frac{a(x - y)^2 - \frac{1}{2} (a - b)(x - y)^2
		+ 2(a - b) \left( z - w - \frac{1}{2} (x - y) \right)^2}{x + y}	
		\\
		& =
		\frac{\frac{1}{2} (a + b) (x - y)^2 
		+ 2(a - b) \left( z - w - \frac{1}{2} (x - y) \right)^2}{x + y}.	
	\end{align*}
	If $a \geq b$, then $\frac{u(x,y,z,w)}{x + y} \geq 0$. 
	Moreover, for $(x,y,z,w) = \left( \frac{1}{4},\frac{1}{4},0,0 \right) \in S$, 
	we have $\frac{u(x,y,z,w)}{x + y} = 0$. 
	This shows that $\beta_{W_{a,b}} = \frac{1}{2}$ whenever $a \geq b$.
	Now assume that $a < b$. 
	Let $(x,y,z,w)$ be an arbitrary element of $S$. 
	If $z - w \geq \frac{1}{2} (x - y)$, then we get
	\[
		\left( z - w - \frac{1}{2} (x - y) \right)^2 
		\leq \left( x - \frac{1}{2} (x - y) \right)^2
		= \left( \frac{x + y}{2} \right)^2.	
	\]
	If $z - w < \frac{1}{2} (x - y)$, then we have
	\[
		\left( z - w - \frac{1}{2} (x - y) \right)^2 
		\leq \left( -y - \frac{1}{2} (x - y) \right)^2
		= \left( \frac{x + y}{2} \right)^2.	
	\]
	Hence, it follows that
	\[
		\frac{u(x,y,z,w)}{x + y} 
		\geq \frac{2(a - b) \left( \frac{x + y}{2} \right)^2}{x + y}
		= \frac{1}{2} (a - b)(x + y) \geq \frac{1}{2} (a - b). 
	\]
	Further, for $(x,y,z,w) = \left( \frac{1}{2},\frac{1}{2},\frac{1}{2},0 \right) \in S$, 
	note that $\frac{u(x,y,z,w)}{x + y} = \frac{1}{2} (a - b)$. 
	Thus, we conclude that
	\[
		\beta_{W_{a,b}} = \frac{1}{2} + \frac{a - b}{2(a + b)} = \frac{a}{a + b},
	\]
	in the case when $a < b$. 
	
	Also, we have
	\[
		\beta_{W_{a,b}}(H,H^c) = \frac{\eta(H \times H) + \eta(H^c \times H^c)}{d}
		= \frac{a}{a + b},
	\]
	and hence, $\beta_{W_{a,b}} = \beta_{W_{a,b}}(H,H^c)$.
\end{proof}

Note that if $a < b$, then $\mathfrak{d}_{W_{a,b}} = 0$.
Using this family of Cayley graphons, we now prove \cref{thm:unbounded-ratios}.

\begin{proof}[Proof of \cref{thm:unbounded-ratios}]
	Fix a positive real number $d$ with $d \leq 1$.
	Let $a$ and $b$ be real numbers satisfying $0 < a < b \leq 1$, and 
	$\frac{1}{2} (a + b) = d$, and $W_{a,b}$ denote the graphon 
	as defined in the beginning of this section.
	Then it follows from \cref{lem:compute-h,lem:compute-beta} that 
	\[
		\frac{h_{W_{a,b}}}{\beta_{W_{a,b}}} = \frac{a + b}{2a}
		= \frac{d}{a},
	\]
	and thus, the set
	\[
		\left\{ \frac{h_{W_{p,q}}}{\beta_{W_{p,q}}} \biggm| 0 < p < q \leq 1,
		\frac{p + q}{2} = d \right\}
	\]
	is unbounded.
	Since $d$ is the degree of the graphon $W_{p,q}$ for any real numbers $p,q$
	satisfying $0 < p < q \leq 1$ and $\frac{1}{2} (p + q) = d$, the result follows.
\end{proof}

\begin{remark}
	Using \cref{lem:compute-h,lem:compute-beta}, it also follows that
	the ratios of the bipartiteness ratios and the Cheeger constants
	of connected non-bipartite Cayley graphons cannot be bounded from above
	by a function of their degrees. 
\end{remark}

\section*{Acknowledgements}

The author is grateful to Jyoti Prakash Saha for suggesting the problem, 
and for his guidance and support through several helpful discussions about this work.
The author acknowledges the fellowship from the University Grants Commission 
with reference number 231620172128.


\end{document}